\documentclass[reqno,11pt,a4paper]{amsart}

\usepackage[english]{babel}
 \usepackage{amsmath,amssymb}
\usepackage{graphicx}
\usepackage{color}
\usepackage{subfigure}
\usepackage{wasysym}
\usepackage{pdfsync}
\usepackage{enumitem}
\usepackage[dvipdfm,a4paper,left=30mm,right=30mm,top=35mm,bottom=35mm]{geometry}

\newcommand{\supp}{\text{supp}}
\newcommand{\I}{\mathcal{I}}

\newcommand{\intone}{\int_{z}^{\bar{z}}}
\newcommand{\inttwo}{\int_0^z}
\newcommand{\intI}{\int_{\mathcal{I}}}
\newcommand{\ti}{\tilde}
\newcommand{\al}{\alpha}
\newcommand{\be}{\beta}
\newcommand{\ga}{\gamma}
\newcommand {\sg} {\sigma} 

\newcommand {\vi}{\varphi} % \vphi

\newcommand{\la}{\label}
\newcommand{\p}{\partial}
\newcommand{\re}{\eqref}

\newcommand{\R}{\mathbb{R}}

\newcommand {\f}   {\frac}

\newcommand{\beq}{\begin{equation}}
\newcommand{\eeq}{\end{equation}}

\newtheorem{thm}{Theorem}[section]
\newtheorem{prop}{Proposition}[section]
\newtheorem{lem}[thm]{Lemma}
\newtheorem{rem}[thm]{Remark}

 \title{On a Boltzmann mean field model for knowledge growth} 

\author{Martin Burger}
\address{Institute for Computational and Applied Mathematics, University of M\"unster, Einsteinstrasse 62, 48149 M\"unster, Germany}
\email{martin.burger@wwu.de} 
\author{Alexander Lorz}
\address{Sorbonne Universit\'es, UPMC Univ Paris 06, UMR 7598, Laboratoire Jacques-Louis Lions, F-75005, Paris, France; CNRS, UMR 7598, Laboratoire Jacques-Louis Lions, F-75005, Paris, France; INRIA-Paris-Rocquencourt, EPC MAMBA, Domaine de Voluceau, BP105, 78153 Le Chesnay Cedex}
\email{alexander.lorz@upmc.fr} 
\author{Marie-Therese Wolfram}
\address{Radon Institute for Computational and Applied Mathematics, Austrian Academy of Sciences, Altenbergerstr. 69, 4040 Linz, Austria}
\email{mt.wolfram@ricam.oeaw.ac.at}

\begin{document}
\maketitle

\begin{abstract}
In this paper we analyze a Boltzmann type mean field game model for knowledge growth, which was proposed by
Lucas and Moll \cite{LM2013}. We discuss the underlying mathematical model, which
consists of a coupled system of a Boltzmann type equation for the agent density and a Hamilton-Jacobi-Bellman equation for
the optimal strategy. We study the analytic features of each equation separately and show local in time existence and uniqueness for
the fully coupled system. Furthermore we focus on the construction and existence of special solutions, which relate to exponential growth
in time - so called balanced growth path solutions. Finally we illustrate the behavior
of solutions for the full system and the balanced growth path equations with numerical simulations.\\
\end{abstract}

\pagestyle{myheadings}
\thispagestyle{plain}
\markboth{M.~BURGER AND A. LORZ AND M.T. WOLFRAM}{A BMFG MODEL FOR KNOWLEDGE GROWTH}

\section{Introduction}

\noindent 
Endogenous growth theory is based on the assumption that human capital, innovation and knowledge are significant factors for economic growth. Understanding how innovation
and knowledge lead to long-term economic growth attracted a lot of interest in the macroeconomic literature. Different models have been proposed to describe knowledge
increase; most of them relate innovation and/or imitation to knowledge growth. In \cite{JN1996,JR1989} economic growth is initiated by imitation in random meetings.
Koenig et al. \cite{KLZ2012} proposed a decision based models, in which firms decide between 'imitation' or 'innovation'. Luttmer \cite{L2012, L2012-2} claims that noise and imitation
lead to growth and considered a model in which individuals are characterized by their productivity or knowledge level. Here the cumulative distribution function describes
the distribution of knowledge, which evolves as individuals meet each other, compare ideas and improve their own knowledge level. Individual meetings are modeled by 
'collisions' in a Boltzmann type equation for the distribution function. The interaction frequency among individual agents 
is assumed to be given; agents do not decide how much time they invest in learning or working. Lucas and Moll \cite{LM2013}
extended this approach by modeling agents as rational individuals, which decide between either option. Each agent's decision is based on maximizing its individual earnings
given the distribution of all other agents with respect to their knowledge level.  The resulting system corresponds to a Boltzmann equation for the agent distribution that is coupled to an
Hamilton-Jacobi-Bellman equation for the optimal strategy. This novel approach combines mathematical models developed in the field of kinetic equations as well
as mean field games.\\ 

\noindent Mathematical tools and methods from statistical mechanics such as kinetic theory have become a popular and successful tool in economics and social sciences.
The kinetic theory developed by Boltzmann studies the statistical behavior of a system not in equilibrium and has its origins in analyzing the thermodynamics of dilute gases.
The Boltzmann equation describes the evolution of the probability distribution function of molecules due to microscopic interactions, namely the collisions between particles.  
In socio-economic applications these collisions correspond to trading events cf. \cite{BCMW2013,DMT2008}, the exchange of opinion cf. \cite{T2006,DMPW2009, NPT2010,PT2014, BS1,BS2} or non-cooperative games, cf. \cite{DLR2014}. 

\noindent In mean field models the effect of all other individuals on a single individual is replaced by an averaged effect. In mean field game theory the dynamics of a single
individual are determined from a stochastic optimal control
problems, in which the optimal strategy is influenced by the knowledge of the distribution of all other players at all times. Mean field game models received an increasing interest in economics, e.g. for describing the strategic decision making
processes in large player stochastic games, cf. \cite{LL2007, HMC2006} in the last years. Their general structure corresponds to a coupled system of a Fokker-Planck equation describing the evolution
of the macroscopic agent density (forward in time) and a Hamilton-Jacobi-Bellman equation for the optimal strategy (backward in time). \\
Degond et al. \cite{DLR2014} introduced a different mean-field kinetic model for rational agents that act in a game-theoretical framework. This framework initially developed to describe
herding, has been extended to model wealth evolution cf. \cite{DLR2014-2} and further investigated with respect to model predictive control, cf. \cite{DHL2015}. \\

\noindent In this paper we investigate existence and uniqueness of solutions to a Boltzmann mean field game (BFMG) model for knowledge growth, with
a particular focus on the construction of so called balanced growth path. The BMFG system exhibits interesting analytic features, such as mass accumulation of agents at the maximum initial knowledge level.
 Balanced growth path are special solutions, for which the cost function, here the overall production, grows 
exponentially in time.  We discuss and identify specific assumptions and conditions, which allow for the existence of such solutions and illustrate the behavior with numerical simulations.

\noindent This paper is organized as follows: in section \ref{s:modeling} we present the Boltzmann mean field game model of Lucas and Moll and discuss specific modeling assumptions.
Section \ref{s:analysis} focuses on the separate analysis of the Boltzmann and the Hamilton-Jacobi-Bellman equation as well as the coupled system. In section \ref{s:bgp} we discuss the
existence of balanced growth path solutions. We conclude by presenting various numerical examples to illustrate the behavior of the proposed model in 
section \ref{s:numerics}.

\section{A Boltzmann type model for knowledge growth}\label{s:modeling}

\noindent Lucas et al. \cite{LM2013} studied the following scenario: consider an economy or society with a constant number of interacting
agents, which are characterized by their knowledge level. Agents can decide how much time they devote to working (by producing goods with
the knowledge the already have) and how much to learning (by exchanging ideas with others).
Each agent is characterized by its knowledge level $z \in \mathcal{I}$ and the fraction of time $s=s(z,t): \mathcal{I} \times [0,T] \to [0,1]$ it devotes to learning. The interval $\mathcal{I}
$ may correspond to the positive real line, i.e. $ \mathcal{I} = \mathbb{R}^+$, or the bounded interval with maximum
knowledge level $\bar{z}$, i.e. $\mathcal{I} = [0, \bar{z}]$. The function $f = f(z,t)$ 
describes the distribution of the agents with respect to their knowledge $z$ and time $t$. We assume that each agent has one unit of time, hence the time devoted to working corresponds to $1-s(z,t)$.  \\

\noindent The evolution of the distribution $f = f(z,t)$ is modeled by a Boltzmann type approach; individuals meet (i.e. collide) and exchange ideas and knowledge. Lucas and Moll proposed the following minimal interaction law to model knowledge increase. 
If two individuals with knowledge levels $z$ and $z'$ meet,
their post-collision knowledge corresponds to
\begin{align}\label{e:collision}
z = \max(z, z'),
\end{align}
i.e. the agent with the lower knowledge level matches its level with the other. Based on \eqref{e:collision}, the evolution of all agents $f = f(z,t)$ can be described by the 
following Boltzmann type equation:
\begin{subequations}\label{e:boltzmann_moll}
\begin{align}
\partial_t f(z,t) &=  -\alpha(s(z,t)) f(z,t) \intone f(y,t) dy  + f(z,t) \int_0^z \alpha(s(y,t)) f(y,t) dy\\
f(z,0) &= f_0(z).
\end{align}
\end{subequations}
The function $\alpha = \alpha(s(z,t)):[0,1]\rightarrow [0,1]$ denotes the learning function, i.e. the interaction probability of an individual with knowledge level $z$.
 We assume that the initial distribution of all agents
 $f_0 = f_0(z)$ is a probability distribution, i.e. it satisfies $\intI f_0(z)~dz = 1$ and $f_0(z) \geq 0$ for all $z \in \I$. 
 \noindent Note that equation \eqref{e:boltzmann_moll} can be written in terms of the Heaviside function $H = H(z)$, namely
\begin{align}\label{e:boltzmann}
\partial_t f(z,t) = -\alpha(s(z,t)) f(z,t) ((1-H)*f) + f(z,t) (H*(\alpha f)).
\end{align}
Another reformulation of \eqref{e:boltzmann} is based on the cumulative distribution function (CDF) $F(z,t)=\int_0^z f(x,t)\,dx$ and reads as 
\begin{align}\label{e:F}
\partial_t F(z,t) = -[1-F(z,t)]\int_0^z \alpha(s(x,t)) f(x,t) dx.
\end{align}
\noindent We assume that the working and learning time is directly related to the individual benefit; more precisely that the earnings $y = y(z,t)$ of an agent with knowledge level $z$ 
correspond to the product of the time the individual spends on working,  i.e. $1-s(z,t)$, times its knowledge level $z$. Hence we have
\begin{align}\label{e:earning}
y(z,t) = (1-s(z,t))z.
\end{align} 
The per-capita production illustrates the total earning in a society and is given by
\begin{align}\label{e:production}
Y(t)  = \int_{\mathcal{I}}(1-s(z,t)) z f(z,t)~dz.
\end{align}
\noindent Based on the per-capita production  each agent wants to maximize its earnings (discounted by a given temporal
discount factor $r \in \mathbb{R}^+$), by choosing an optimal partition of its working respectively learning time. Then the optimal ratio of
working to learning time (a quantity related to the work-life balance), is determined by the solution $s = s(z,t)$ of the optimal control problem
\begin{align*}
  V(x,t') = \max_{s \in \mathcal{S}} \bigl[\int_{t'}^T \int_0^{\bar z} {e^{-r(t-t')}} (1-s(z,t)) z  \rho_x(z,t) dzdt \bigr],
\end{align*}
subject to 
\begin{align*}
\partial_t \rho_{x}(z,t) = - \alpha(s) \rho_x(z,t) \int_z^{\bar{z}} f(y,t)\,dy + f(z,t) \int_0^z\alpha(s) \rho_x(y,t)\,dy
\end{align*}
with $\rho_x(z,t') = \delta_x$.
Here $\mathcal{S}$ denotes the set of admissible controls given by
\begin{align}
\mathcal{S} = \lbrace s:\mathcal{I} \times [0,T] \to [0,1]\rbrace.
\end{align}

\noindent Then the optimal strategy can be calculated via the Lagrange functional
\begin{align*}
  \mathcal{L} = &\int_{t'}^T\int_0^{\bar z} e^{-r(t -t')} (1-s(z,t)) z \rho_x(z,t)  \\
  &~ - \left[\partial_t \rho_x(z,t) + \alpha(s(z,t)) \rho_x(z,t) \int_z^{\bar z} f(y,t)\,dy \right.\\
&~ - \left.  f(z,t) \int_0^z\alpha(s(y,t)) \rho_x(y,t)\,dy\right] e^{-r(t -t')} V(z,t)\, dz dt.
\end{align*}
The optimality condition with respect to $f$ corresponds to the Hamilton-Jacobi-Bellman equation:
\begin{align*} 
 \partial_t V(z,t) - rV(z,t)  + \max_{s \in \mathcal{S}} &\left[(1-s(z,t)) z - \alpha(s(z,t)) V(z,t) ((1-H)*f) \right.\\
 &\left. + \alpha(s(z,t)) ((1-H)*(Vf))\right]= 0.
 \end{align*}
The function $V = V(z,t)$ denotes the value function, and represents the production level starting from knowledge level $z$ at time $t$ and controlling the
system until a final or infinite time. Altogether we obtain a Boltzmann mean field game (BMFG) of the form
\begin{subequations}\label{e:bmfg1}
\begin{align}
 & \partial_t f(z,t) = -\alpha(S(z,t)) f(z,t) ((1-H)*f) + f(z,t) \Bigl(H*\bigl(\alpha(S(z,t)) f\bigr)\Bigr), \label{e:boltz}\\
\begin{split}
&\partial_t V(z,t) - rV(z,t) =  - \left[(1-S(z,t)) z - \alpha(S(z,t)) V(z,t) ((1-H)*f)\right.\\
&\qquad\qquad  +\left. \alpha(S(z,t)) ((1-H)*(Vf)) \right], \label{e:bellmann}
\end{split}\\
\begin{split}
&S(z,t) = \arg \max_{s \in \mathcal{S}} \left[(1-s(z,t)) z - \alpha(s(z,t)) V(z,t) ((1-H)*f) \right.{}\\
 &\qquad\qquad+ \left.\alpha(s(z,t)) ((1-H)*(Vf)) \right],\label{e:S}
\end{split}\\
&f(z,0) = f_0(z),\\
&V(z,T) = 0.
\end{align}
\end{subequations}
\noindent Lucas and Moll refer to $\alpha = \alpha(s):[0,1] \to \R^+$ as the learning technology function, which may have the form 
\begin{align}\label{e:alpha}
\alpha(s) = \alpha_0 s^n, \quad n \in [0,1).
\end{align}
%We denote the optimal control by $S = S(z,t)$, i.e. 
%\begin{align}\label{e:S}
%\begin{split}
%S(z,t) = \arg \max_{s \in \mathcal{S}} \left[(1-s(z,t)) z - \alpha(s(z,t)) V(z,t) ((1-H)*f) \right.{}\\
% + \left.\alpha(s(z,t)) ((1-H)*(Vf)) \right].
%\end{split}
%\end{align}

\noindent We conclude this section by stating three special modeling situations, which we shall discuss
and analyze later.

\subsection{Symmetric meetings}
In the initial model only one party in the interaction/collisions gains knowledge, the other one
has no benefit. To capture symmetric meetings Lucas and Moll propose a modified Boltzmann type equation of the form
\begin{align}\label{e:symmetric}
\begin{split}
\partial_t f(z,t) =& -f(z,t) \int_z^\infty[\alpha(s(z,t)) + \beta \alpha(s(y,t))] f(y,t) dy \\
&+f(z,t) \int_0^z [\alpha(s(y,t)) + \beta \alpha(s(z,t))] f(y,t) dy,
\end{split}
\end{align}
where $\beta \in [0,1]$ encodes the probability to learn from each other. The case $\beta = 0$ corresponds to the original
model, $\beta = 1$ to perfectly symmetric meetings.

\subsection{Special case: $\alpha = \alpha_0 \in \R^+$} 
Let us consider the BMFG model \eqref{e:bmfg1} with a given constant learning function $\alpha = \alpha_0 \in \R$, i.e. $n = 0$ in \eqref{e:alpha}. Note that in this case the
value function $V = V(z,t)$ is positive by definition and that $f = f(z,t)$ is non-negative for all times if $f(z,0) \geq 0$. Then the maximum of
\begin{align*}
((1-s(z,t))z + \alpha_0 \left(V(z,t) ((1-H) * f) + ((1-H) * (Vf))\right) 
\end{align*}
is given by $s(z,t) = 0$. In this case system \eqref{e:bmfg1} decouples and the Boltzmann type equation \eqref{e:boltz} can be written in terms of the cumulative distribution function $F = F(z,t)$, i.e.
\begin{align}\label{e:cdf}
\partial_t F(z,t) = - \alpha_0 (1-F(z,t)) F(z,t).
\end{align}
Hence the optimal strategy $V = V(z,t)$ can be calculated independently, which also motivates the separate analysis of the Boltzmann and the HJB equation in the next section.

\subsection{Balanced growth path}
Lucas and Moll postulated the existence of balanced growth path (BGP) solutions to system \eqref{e:bmfg1}, for which the production rate \eqref{e:production} grows exponentially in time. 
BGPs correspond to solutions in the rescaled variables $(\phi, v, \sigma)$:
\begin{align}\label{e:rescal}
&f(z,t) =  e^{-\ga t}\phi(ze^{-\ga t}), ~~V(z,t)= e^{\ga t}v(ze^{-\ga t}) \text{ and }s(z,t)= \sg(ze^{-\ga t})
\end{align}
assuming the existence of a constant $\gamma \in \mathbb{R}^+$. If this is the case the production \eqref{e:production} can be transformed to
\begin{align*}
Y(t) = e^{\gamma t} \int_0^\infty[1-\sigma(x)] x \phi(x) dx.
\end{align*}
The rescaled equations for $(\phi, v, \sigma) = (\phi(x), v(x), \sigma(x))$ read as 
\begin{subequations}\label{e:bgp}
\begin{align}
&-\ga \phi(x) - \ga \phi'(x) x = \phi(x)\int_0^x\al(\sg(y))\phi(y)\,dy - \al(\sg(x))\phi(x)\int_x^\infty\phi(y)\,dy,\la{e.phi}\\
&(r-\ga)v(x)+\ga v'(x) x = \max_{\sg\in\Xi}\left\{(1-\sg)x+\al(\sg)\int_x^\infty[v(y)-v(x)]\phi(y)\,dy \right\}, \la{e.v}
\end{align}
\end{subequations}
where $\Xi = \lbrace \sigma: \R^+ \to [0,1] \rbrace$ denotes the set of admissible controls.

\section{Analysis of the Boltzmann mean field model}\label{s:analysis}

\noindent In this section we present a local existence result for the fully coupled Boltzmann mean field model \eqref{e:bmfg1}.
We start with the analysis of the Boltzmann equation \eqref{e:boltzmann_moll} for a given interaction rate $\alpha = \alpha(z,t)$. Then we 
discuss existence and uniqueness of solutions for the Hamilton Jacobi equation and conclude the section by studying the fully coupled system.\\

\noindent In this section we will need the following assumptions on the initial datum $f_0 = f_0(z)$:
\begin{enumerate}[label=(A\arabic*)]
\item \label{a:f0prob}  Let $f_0 \in L^{\infty}(\mathcal{I})$ be a probability density, i.e. $\intI f_0(z) dz = 1$ and $f_0(z) \geq 0$ for all $z \in \I$. 
\end{enumerate}

\subsection{Analysis of the Boltzmann equation for a given learning function $\alpha$}\label{s:boltzmann}
We start with the analysis of the Boltzmann type equation \eqref{e:boltzmann_moll} for a given learning function $\alpha = \alpha(z,t)$. Hence we consider 
\begin{subequations}\label{e:boltzf}
\begin{align}
\partial_t f(z,t) &= -\alpha(z,t) f(z,t) \intone f(y,t)~dy + f(z,t) \inttwo \alpha(y,t) f(y,t) dy,\\
f(z,0) &= f_0(z),
\end{align}
\end{subequations}
on the interval $\I = [0, \bar{z}]$.

\begin{rem}\label{r:adjoint}
Let us introduce the operators $G_s = \alpha(s(z,t)) \intone f(y,t)~ dy$ and \\$L_s = \inttwo \alpha(s(y,t)) f(y,t)~ dy$. Then equation 
\eqref{e:boltzf} can be written as 
\begin{align*}
\partial_t f(z,t) =  f(z,t) \left( L_s f(z,t)-G_s f(z,t)  \right).
\end{align*}
Note that $G_s$ and $L_s$ are adjoint operators, i.e. $L_s^* = G_s$ since
\begin{align*}
 (L_sf, g)  &= \int_0^{\bar{z}}\bigl[ \inttwo \alpha(s(y,t)) f(y,t) dy\bigr] g(z,t) dz\\
& = \int_0^{\bar{z}} \bigl[\int_y^{\bar{z}} g(z) dz\bigr] \alpha(s(y,t)) f(y,t) dy.
\end{align*}
\end{rem}

\noindent First we present a global existence result for \eqref{e:boltzf} in time.

\begin{thm}
Let \ref{a:f0prob} be satisfied and $\alpha = \alpha(z,t) \in L^1(\mathcal{I}) \times L^{\infty}([0,T])$. Then equation \eqref{e:boltzf} has a global in time solution $f = f(z,t) \in L^1(\mathcal{I}) \times L^\infty([0,T])$.
\end{thm}
\begin{proof}
Let $T>0$ be given, $\bar \al: = \max \al$   and $A$ the following closed subset:
$$A=\{g\in C([0,T],L^1(\mathcal{I})), g \ge 0, \| g(\cdot,t) \|_1 \le a\},$$
where $a = c \int f_0(y)\,dy $. For each $g \in A$ we define the operator $\Phi(g)$ as the solution of 
 \begin{align}\label{e:boltzf_g}
\partial_t f(z,t) = -\alpha(z,t) f(z,t) \intone g(y,t)~dy + f(z,t) \inttwo \alpha(y,t) g(y,t) dy,
\end{align}
with initial data $f(z,0) = f_0(z)$. Then the existence of a solution follows from Picard Lindeloef by showing that the operator $\Phi$
\begin{enumerate}[label=(\alph*)]
\item maps $A$ onto itself,
\item is a contraction for $T$ small.
\end{enumerate}
A priori estimates: The change of the total mass can be estimated by  
$$\frac{d}{dt}\int_\mathcal{I} f(y,t) \,dy \le \bar{\alpha} \|g\|_1 \int_\mathcal{I} f(y,t)\,dy .$$
Based on this estimate we use Gronwall to deduce that
\begin{align*}
\int_\mathcal{I} f(y,t) \,dy \le\exp( \bar{\alpha}\int_0^t \|g\|_1\,ds) \int_\mathcal{I} f_0(y) \,dy. 
\end{align*}
Hence (a) is satisfied.  To show that $\Phi$ is a contraction we consider equation \eqref{e:boltzf_g} for two given functions $g_1$ and $g_2$, $g_i \in A$, $i=1,2$.
Then their difference satisfies 
\begin{align*}
\p_t (f_1 - f_2)(z,t)= &-\al(z,t) f_1(z,t) \intone g_1(y,t) \,dy +f_1(z,t) \inttwo \alpha(y,t) g_1(y,t) \,dy\\
 &+ \al(z,t) f_2(z,t) \intone g_2(y,t) \,dy - f_2(z,t) \inttwo \alpha(y,t) g_2(y,t) \,dy.
\end{align*}
Using $\|f_2\|_1 \le \|f_0\|_1 e^{a\bar \al T}=:K_{loc}$, for the $L^1$-norm of the difference,  we obtain the differential inequality
$$\p_t \| f_1 - f_2\|_1 \le 2a \bar \al \| f_1 - f_2\|_1  + 2  K_{loc} \bar \al \|g_1-g_2\|_1.$$
Since $f_1(z,0)= f_2(z,0)$ we deduce, using Gronwall, that $\Phi$ is a contraction for sufficiently small time $T$. Note that all constants in the local existence
argument depend on the initial mass only. Since we have mass conservation, i.e.
\begin{align*}
\intI f(z,t)\,dz = \intI f_0(z)\,dz \text{ for all } t > 0,
\end{align*}
we can iterate the local existence argument at $T, 2T, \ldots$ to obtain global existence.
\end{proof}

\noindent Next we show that the support of $f$ remains bounded if the initial datum $f_0$ has a compact support.
\begin{prop}\label{l:compactsupp}
Let $\alpha \in C([0,T) \times \I) $ and $f = f(z,t)$ be a continuous solution to \eqref{e:boltzf}, i.e. $f \in C([0,T) \times \I)$ with 
$\supp f(\cdot, 0) \subset [0,M]$, $M < \bar{z}$. Then
\begin{align*}
\supp (f(\cdot, t)) \subset [0,M] \text{ for all times } t > 0.
\end{align*}
\end{prop}
\begin{proof}
The proof is based on the maximum principle. Let us assume there exists a point $\hat{z}$ in $(M, \bar{z}]$ such that $f(\hat{z}, t) > 0$. Then
$0 < \partial_t f(\hat{z},t)$. But since $f(z,0) = 0$ for all $z > M$, we deduce that $f(z, t) = 0$ for all $z > M$. \lightning
\end{proof}

\noindent Note that Lemma \ref{l:compactsupp} is only valid for continuous solutions. However we expect that solutions of \eqref{e:boltzf} converge to 
Dirac deltas as time evolves. This can be explained by the fact that individuals with a lower
knowledge level gain knowledge in each collision, but individuals with the greater knowledge level cannot improve. Hence we
conjecture the formation of Dirac deltas at $z = M$, if $\supp(f) \subset [0,M]$. This can be observed in the evolution
of the first order moment, i.e.
\begin{align}
\frac{d}{dt} \intI z f(z,t) dz &= \intI \bigl[ z f(z,t)  \inttwo \alpha(y,t) f(y) dy - f(z,t) \intone \alpha(z,t) f(y,t) z dy] dz \nonumber \\
&= \intI f(z) \intone \alpha(y,t) f(y,t)\underbrace{(y-z)}_{\geq 0} dy dz,\la{e.firstmoment}
\end{align}
where we used the fact that $L_s$ and $G_s$ are adjoint operators (see Remark \ref{r:adjoint}). Hence we deduce that the first order moment is increasing in time. 
Also the mass located in the interval $(z_0,\bar z)$ is increasing, because
\begin{align}
\frac{d}{dt} \int_{z_0}^{\bar{z}}  f(z,t) dz &= \int_{z_0}^{\bar{z}} \bigl[  \int_0^z f(z,t) f(y,t) \alpha(y,t) dy - \int_z^{\bar{z}}  f(z,t)\alpha(z,t) f(y,t) dy \bigr]  dz\nonumber\\
&= \int_{z_0}^{\bar{z}}  \bigl[ \int_0^z f(z,t) f(y,t) \alpha(y,t) dy - \int_{z_0}^{\bar{z}}  f(z,t) f(y,t) \alpha(y,t) dy \bigr] dz \nonumber\\
&= \int_{z_0}^{\bar{z}}  \int_0^{z_0} f(z,t) f(y,t) \alpha(y,t) dy dz \geq 0. \label{e:mass}
\end{align}
From the previous estimates we deduce the following theorem.
\begin{thm}\label{l:deltadirac}
Let $\alpha(z,t) \ge \underline \alpha >0$ and $\bar z \in \supp(f)$, then 
\begin{align*}
f(\cdot,t) \rightharpoonup^* \delta_{\bar z}.
\end{align*}
\end{thm}
\begin{proof}
Setting $z_0 = 0$ in equation \eqref{e:mass} gives $\frac{d}{dt} \int_{z_0}^{\bar{z}}  f(z,t) dz=0$ and therefore implies mass conservation. 
From \eqref{e:mass} we deduce that 
\begin{align*}
-\frac{d}{dt}F(z_0,t)= \frac{d}{dt}(1-F(z_0,t)) &\ge  \underline \alpha \int_{z_0}^{\bar{z}} f(z,t)dz F(z_0,t) =   \underline \alpha (1-F(z_0,t)) F(z_0,t).
\end{align*}
This differential inequality implies that the CDF $F(z_0,t) \to 0$ for $z_0 < \bar z$ as time $t\to \infty$. Since we can choose $z_0$ close to $\bar z$ we conclude that $f$ converges to a Dirac mass.
\end{proof}

\begin{rem}
Note that the formation of a Delta Dirac mass accumulates at $z = \tilde{z}$, where $\tilde{z} = \max_{z} \supp(f)$ for compactly supported initial datum and $\tilde{z} = \bar{z}$ for positive initial data
on the bounded domain $\I = [0,\bar{z}]$. If $f_0(z) > 0$ for all $z \in \R^+$ the mass accumulates at $z = \infty$.
\end{rem}

\subsection{Analysis of the Hamilton-Jacobi Bellman equation}

Next we study the analytic behavior of the Hamilton-Jacobi-Bellman equation for a given $f \in C(0,T,L^1)$ on $\I = \R^+$:
\begin{subequations}\label{e:hjbsolo}
\begin{align}
\begin{split}
  \partial_t V(z,t) - rV(z,t) &=  -\max_{s \in \mathcal{S}} \left[(1-s(z,t)) z - \alpha(s(z,t)) V(z,t) ((1-H)*f)\right.\\
& \phantom{=-\max_{s\in[0,1]}}\left. + \alpha(s(z,t)) ((1-H)*(Vf)) \right]
\end{split}\\
V(z,T) &= 0.
\end{align}
\end{subequations}

\noindent We shall need the following assumption for the terminal value function in the rest of the section:
\begin{enumerate}[label=(A\arabic*), start=2]
\item \label{a:V_pos} Let the final data $V(\cdot,T)$ in equation \re{e:hjbsolo} be non-negative and non-decreasing.
\end{enumerate}

\noindent To ensure the existence of a maximizer we need the following assumptions on the learning function $\alpha = \alpha(s)$:
\begin{enumerate}[label=(A\arabic*), start=3]
\item \label{a:alpha} Let $\alpha: [0,1]  \to \R^+$, $\alpha \in C^{\infty}([0,1]),~\alpha(0) = 0,~\alpha'(0) = \infty$, $\alpha''<0$ and $\alpha$ monotone.
\end{enumerate}
In the following we shall use the variable 
\begin{align*}
B = -V(z,t) (1-H) * f + (1-H) * (Vf), 
\end{align*}
to enhance readability.

\begin{lem} \label{le.unique_s} Let assumption \ref{a:alpha} be satisfied, $z >0$ and $B \in \R$. Then there exists a unique solution $ S =  S (B)$ of the optimization problem 
\beq
\max_{s\in\mathcal{S}}\left((1-s)z +\al(s)B\right), \label{e.maxB}
\eeq
with $S = \arg \max_{s\in\mathcal{S}}\left((1-s)z +\al(s)B\right)$.
\end{lem}
\begin{proof}
Let $\be := \al^{-1}$, then the problem with $\zeta = \al(s)$ is equivalent to 
$$\arg \max_{\zeta \in [0,\al(1)]} z(1-\be(\zeta) )+ B \sg$$
Because of the strict concavity of $-\be$, there exists a unique solution.
\end{proof}

\begin{lem}\label{le.lipschitz} Let assumption \ref{a:alpha} be satisfied, $z >0$, $B \in \R$ and $S=S(B)$ be the optimal control satisfying \re{e.maxB} for a given $B$. If
\begin{align*}
\lim_{B\to0}\al''(S(B))B^3<0,
\end{align*}
then the maps
$B \to S(B)$, $B\to \al(S(B))$ and $B \to \al(S(B))B$ are Lipschitz.
%$\al:\mathbb{R}^+ \to \R$ with $\al(0)=0, \al'(0)=\infty,$ $C^2$ on $(0,\infty)$, $\al''<0$, %$\al''$ bounded on $[\epsilon,\infty)$, 
%$\al$ monotone, $A >0, B \in \R$. Then there exists a unique solution $ s =  s (B)$ of the optimisation problem $$\max_{s\in[0,1]}(1-s)A +\al(s)B$$.% and the map $\mk \to  \bar s (\mk)$ is Lipschitz continuous.
\end{lem}
\begin{proof} We distinguish between the following three cases:\\
%\begin{enumerate}[label=(\alph*)]
Case 1: If $B<0$, then $s=0$.\\
Case 2: If $B > \f{z}{\al'(1)}$ and since $\al$ is concave we deduce that
\begin{align*}
z(1-s)+B\al(s)&\le z(1-s) + B\al(1) + B \al'(1)(s-1)\\
&= B\al(1) + (1-s)(z-B\al'(1))<B\al(1). 
\end{align*}
Hence the maximum is attained at $s=1$.\\
Case 3: If $0<B<\f{z}{\al'(1)}$ let $\al'(s) = \f{d}{ds} \al(s)$ and $\al''(s) = \f{d^2}{ds^2} \al(s)$. In this case there exists a unique solution of $\f{d}{ds}\al(s)=\f z B$, which gives the maximum. 
Furthermore we have that 
$$\f{d^2}{ds^2}\al(s)s'(B) = -\f{z}{B^2} \Rightarrow s'(B) = -\f{z}{B^2\al''(s) },$$
and
\beq
\f{d}{dB}\al(s)= \f{d}{ds}\al(s)s'(B)= \al'(s) s'(B) = \f z B s'(B)=-\f{z^2}{B^3\al''}=-\f{\al'^3}{z\al''}.\la{e.dB_al}\\
\eeq
Because $\lim_{B\to0} -B^3\al''(B) >0$, we conclude that $S$ is piecewise continuously differentiable. Since  
\begin{align*}
\lim_{B\to0} -B^2\al''(B) = \infty,
\end{align*}
we deduce that $S'(B) \rightarrow 0$ as $B \rightarrow 0$. 
Hence we have continuity at $B=0$ and $B=\f{A}{\al'(1)}$, Lipschitz continuity follows for $S(B)$. The same is true for $\al(S)B$. %\qedhere
%\end{enumerate}
\end{proof}

\begin{lem}\la{l.V_pos}
Let assumption \ref{a:V_pos} be satisfied. Then the value function $V(\cdot,t)$ solving \eqref{e:hjbsolo} is non-negative and non-decreasing for all times $t \in [0,T)$.
%Moreover, if the final data $V(\cdot,T)$ is non-decreasing,  then the function $V(\cdot,t)$ remains non-decreasing for all times $t \in [0,T)$.
\end{lem}

\begin{proof}
%\red{
Equation \re{e:hjbsolo} can be written as
\begin{align}
\begin{split}
  \p_t V(z,t)&-rV(z,t)=-(1-S(z,t))z{}\\
&-\al(S(z,t))\left[\int_{z}^\infty V(y,t)f(y,t) dy -V(z,t)\int_{z}^\infty f(y,t)dy\right].\la{e:hjbsolo_integral}
\end{split}
\end{align}
In a minimum point $z_0$ we have 
$$\p_t V(z_0,t)-rV(z_0,t)\le-(1-S(z_0,t))z_0-\al\left[V(z_0,t)\int_{z_0}^\infty f(y,t)dy -V(z_0,t)\int_{z_0}^\infty f(y,t) dy\right].$$
Since this inequality is backwards, it preserves non-negativity.
\newline
Calculating the derivative of equation \re{e:hjbsolo_integral} with respect to $z$, we obtain
\begin{align*}
\p_t V_z(z,t)-rV_z(z,t)&=S_z(z,t)z + S(z,t) -1 \\
&-\al'(S(z,t)) S_z(z,t) B +\al (S(z,t)) V_z(z,t)  \int_z^\infty f(y,t)dy. 
\end{align*}
Using $ \alpha'=z/B$, it follows that
$$\p_t V_z(z,t)-rV_z(z,t)=S(z,t) -1  +\al(S(z,t)) V_z(z,t)  \int_z^\infty f(y,t)dy. $$ 
Since $S\le 1$ we deduce that $V_z$ stays non-negative.
\end{proof}

\begin{lem}\la{l.B_mono}
Let assumption \ref{a:V_pos} be satisfied. Then $B(\cdot,t)$ is non-increasing and the maximizer $S(\cdot,t)$ is  non-increasing for all times $t \in [0,T)$.
Moreover the function  $S(\cdot,t)$ is strictly decreasing on the interval where $0<S(\cdot,t)<1$, except in the degenerate case of $B(0,t)=0$. Then $S(z,t)=1$ for small $z$.
\end{lem}

\begin{proof}
We have already seen above that 
$B(z,t)\ge 0$ and $$B_z=-V_z(z,t)  \int_z^\infty f(y,t)dy \le 0.$$ If $B(0,t)=0$, then $S(z,t)=1$ for all $z$. Otherwise $B(0,t)>0$, so $B > \f{z}{\al'(1)}$ and therefore the maximizer $S(z,t)$ is equal to $1$ for small $z$. If $0<B<\f{z}{\al'(1)}$, we have  $S(z,t)=(\al')^{-1}\left(\f z B\right)$. Since $\al$ is strictly concave,  $S(z,t)$ is strictly decreasing.
\end{proof}

\noindent The previous results lead to the following existence and uniqueness theorem for the Hamilton-Jacobi-Bellman equation.

\begin{thm}
Let $f \in C(0,T,L^1)$ be given and $\al = \al(s,t)$ satisfy assumption \ref{a:alpha}. Then there exists a unique solution $V \in C(0,T,L^\infty)$ of \eqref{e:hjbsolo} with $V(z,T)=0$. Moreover, let $\ti V$ be a solution of \eqref{e:hjbsolo} with $\ti f$,  then there exist constants $m$ and $D$ (independent of $\ti V$  and $\ti f$) such that 
\beq
\|V-\ti V\|_\infty \le D e^{mt}\|f-\ti f\|_{C(0,T,L^1)}\|\ti V\|_\infty.
\eeq
\end{thm}
\begin{proof}The proof is based on the following statements: 
\begin{itemize}
\item[(i)] $V\to B(V,f)=(1-H)\star (Vf)-V(1-H)\star f$ is Lipschitz on the spaces given. 
\item[(ii)] $B(V) \to S$ and $B(V)\to \al(S)B$ are Lipschitz because of Lemma \re{le.lipschitz}. 
\end{itemize}
Then equation  \eqref{e:hjbsolo} is of the form
$$\p_t V= R(V),$$ with $R$ Lipschitz and we can conclude the proof with Picard-Lindeloef.
For  the difference $V-\ti V$ we obtain
$$\p_t (V-\ti V)=R(V)-\ti R(\ti V)= R(V)-R(\ti V) + R(\ti V) -\ti R(\ti V).$$
Since the estimate $\|B(\ti V,f)-B(\ti V,\ti f)\|_\infty \le \|\ti V\|_\infty \|f-\ti f\|_{C(0,T,L^1)}$ holds, we can show that 
$$\|R(\ti V)-\ti R(\ti V)\|\le K \|f-\ti f\| \|\ti V\|_\infty$$ with $K$ independent of $\ti V,f,\ti f$. Therefore we conclude 
\begin{align*}
\|V-\ti V\|_\infty \le D e^{mt}\|f-\ti f\|_{C(0,T,L^1)}\|\ti V\|_\infty. 
\end{align*}
\end{proof}

\subsection{Analysis of the fully coupled Boltzmann mean field game system}
We show existence and uniqueness of the fully coupled system using a fixed-point argument.\\
\noindent Because of the term $(1-s)z$, we expect linear growth of $V$ in $z$. Therefore we are looking for a solution in the space $C(0,T,L^\infty_{1+z}(\R^+))$, where
$$W^\infty:=L^\infty_{1+z}(\R^+):=\{u=(1+z)w| w \in L^\infty(\R^+)\}\text{ with } \|u\|_{L^\infty_{1+z}}:=\text{ess sup } \f{|u(z)|}{1+z}.$$
To compensate we use a weighted $L^1$-norm for $f$:
$$W^1:=L^1_{\f{1}{1+z}}(\R^+):=\{u=\f{w}{1+z}| w \in L^1(\R^+)\}\text{ with } \|u\|_{L^1_{\f{1}{1+z}}}:=\int |u(z)|(1+z)\,dz.$$

\noindent First we state two lemmas, which provide the necessary estimates and bounds for the local existence proof.
\begin{lem}
Let the initial datum $f_0 = f_0(z)$ satisfy assumption \ref{a:f0prob} and the learning function $\alpha = \alpha(s)$ assumption \ref{a:alpha}. Then every solution $f = f(z,t)$ to \eqref{e:boltzmann} has a bounded first order moment.
\end{lem}
\begin{proof}
We reiterate from \eqref{e.firstmoment} that 
\begin{align*}
\frac{d}{dt} \intI z f(z,t)\, dz &\le  \intI f(z)  \int_z^\infty \alpha(y,t) f(y,t)y\, dy dz \le \bar \al  \intI  f(y,t)y\, dy.
\end{align*}
This gives an exponential bound for the first order moment.
\end{proof}

\begin{lem}
Let $V = V(z,t)$ be a solution to \eqref{e:bellmann}. Then $V$ is bounded in $L^\infty_{1+z}(\R^+)$.
\end{lem}
\begin{proof}
Dividing equation  \re{e:bellmann}  by $1+z$  and changing $t$ to $-t$, we obtain
$$\p_t \f V {1+z} \le  r\f V {1+z} + \f z {1+z} +  \f 1 {1+z}  \al B(V,f)$$
and 
$$B(V,f)\le \int_z^\infty (Vf)(x)dx\le\|V\|_{W^\infty}\|f\|_{W^1}\le C \|V\|_{W^\infty} e^{Kt}.$$
This implies
$$\p_t \f V {1+z} \le  r\f V {1+z} + \f z {1+z} +  \f 1 {1+z}  \bar \al  \|V\|_{W^\infty} C e^{Kt}.$$
\end{proof}

\begin{thm} Let assumptions \ref{a:f0prob}-\ref{a:alpha} be satisfied. If $\lim_{s\to 0} \f{(\al')^3}{\al''} < \infty$, then 
the  fully coupled Boltzmann mean field game system \re{e:bmfg1} on $\I = \R^+$ has a unique local in time solution.
\end{thm}

\begin{proof}
For local existence, we take a $g\in C(0,T,L^1_{\f{1}{1+z}}(\R^+))$ and solve 
\beq
\partial_t V(z,t) - rV(z,t) =  -\max_{s \in \mathcal{S}} \left[(1-s(z,t)) z - \alpha(s(z,t)) B(V,g)\right],\label{e:bellmannloc}
\eeq
for $V$ and $S$. Then we solve
\beq
  \partial_t f(z,t) = -\alpha(S(z,t)) f(z,t) ((1-H)*f) + f(z,t) \Bigl(H*\bigl(\alpha(S(z,t)) f\bigr)\Bigr), \label{e:boltzloc}\\
\eeq
for $f$. Let $\Psi$ denote the operator mapping $g$ to $f$.
We show that $\Psi$ is a contraction by considering two solutions $g$ and $\ti{g}$. Taking the difference of equation \re{e:bellmannloc} for $g$ and $\ti g$ gives
$$\p_t (V-\ti V)= r(V-\ti V)+z(S-\ti S)-(\al(S)B(V,g)-\al(\ti S)B(\ti V,\ti g)).$$
To show that the second term on the RHS is Lipschitz, we calculate
$S'(B)=-\f{z}{B^2\al''}=-\f{(\al')^2}{z\al''}$. 
Similarly for the third term on the RHS we deduce that
$$\f{d}{dB}(\al B)=\al'(S) S'(B) B+ \al = z S' + \al = -\f{(\al')^2}{\al''}+\al.$$
Since both derivatives are  bounded we obtain
$$\p_t (V-\ti V)\le r(V-\ti V)+K|(B(V,g)-B(\ti V,\ti g)|.$$
Let us consider the last term only. Due to the structure of $B$, i.e. $B(V,g)=(1-H)\star(Vg)-V(1-H)\star g$, we work on the part $(1-H)\star(Vg)$ only. The other part can be estimated using similar
arguments. We deduce that
\begin{align}
\begin{split}
(1-H)\star(Vg)-(1-H)\star(\ti V\ti g)&=(1-H)\star[(V-\ti V) g+\ti V (g-\ti g)]\\
&=\int_z^\infty (V-\ti V) g\,dx+\int_z^\infty\ti V (g-\ti g)\,dx\\
&\le\|V-\ti V\|_{W^\infty} \|g\|_{W^1}+\|\ti V\|_{W^\infty} \|g-\ti g\|_{W^1}.
\end{split}
\end{align}
This implies the following inequality:
$$\p_t (V-\ti V)\le r(V-\ti V)+2K\|V-\ti V\|_{W^\infty} \|g\|_{W^1}+\|\ti V\|_{W^\infty} \|g-\ti g\|_{W^1}.$$
So for $\int_0^T \|g-\ti g\|_{W^1}\,dt$ small, $\|V-\ti V\|_{W^\infty}$ is small. Moreover according to Lemma \re{l.B_mono}, $S=\ti S=1$ for small $z$, so the Lipschitz constant coming from \re{e.dB_al} is bounded. Therefore we obtain that also the term $|\al(S)-\al(\ti S)|$ is small. This implies for equation \re{e:bellmannloc} that $\|f-\ti f\|_{W^1}\le a \|g-\ti g\|_{W^1}$ with $a<1$ for $T$ small enough. Hence the operator $\Psi$ is a contraction, which concludes the proof.
\end{proof}

\section{Balanced growth paths}\label{s:bgp}

\noindent In this section we discuss the existence of balanced growth path solutions and the convergence behavior towards them. Balanced growth
paths correspond to solutions $(\phi,v,\sigma)$ for which the production function \eqref{e:production} grows exponentially in time. We reiterate that the 
rescaled equations for the Boltzmann mean field game \eqref{e:bmfg1} in the variables $(\phi, v, \sigma) = (\phi(x), v(x), \sigma(x))$ with
\begin{align}\label{e:rescalbgp}
&f(z,t) =  e^{-\ga t}\phi(ze^{-\ga t}), ~~V(z,t)= e^{\ga t}v(ze^{-\ga t}) \text{ and }s(z,t)= \sg(ze^{-\ga t})
\end{align}
 read as
\begin{subequations}\label{e:bgp2}
\begin{align}
-\ga \phi(x) -\ga \phi'(x) x &= \phi(x)\int_0^x\al(\sg(y))\phi(y)\,dy - \al(\sg(x))\phi(x)\int_x^\infty\phi(y)\,dy\la{e.phi2}\\
(r-\ga)v(x)+\ga v'(x) x &= \max_{\sg\in\Xi}\left\{(1-\sg)x+\al(\sg)\int_x^\infty[v(y)-v(x)]\phi(y)\,dy \right\} \la{e.v2}
\end{align}
\end{subequations}
where $\Xi = \lbrace \sigma: \R^+ \to [0,1] \rbrace$ denotes the set of admissible controls. A necessary prerequisite for the
existence of BGP solutions is the assumption that the initial cumulative distribution function $F(z,0)$ has a Pareto tail, which is given by:  

\begin{enumerate}[label=(A\arabic*), start=4] 
\item \label{a:pareto}  The productivity function $F(z,0) = \int_0^z f_0(y) dy$ has a Pareto tail, i.e. there exist constants $k, \theta \in \R^+$ such that
 \begin{align}
\lim_{z \to \infty} \f{1-F(z,0)}{z^{-1/\theta}} = k. \label{e:paretoF}
\end{align}
Condition \eqref{e:paretoF} in the rescaled variable $\phi$ reads as:
\beq 
\lim_{z\to \infty} \f{\int_z^\infty\phi(y)\,dy}{z^{-1/\theta}}=k. \la{e.pareto}
\eeq
\end{enumerate}

\begin{lem}
Let assumption \ref{a:pareto} be satisfied. Then $F = F(z,t)$ has a Pareto tail with the same decay rate $\theta$ for all times $t\in[0,T]$. 
\end{lem}
\begin{proof}
Note that we can rewrite equation \re{e:F} as
$$\partial_t [1-F(z,t)] = [1-F(z,t)]G(z,t) \qquad \text{ with }G(z,t):=\int_0^z \alpha(s(x,t)) f(x,t) dx.$$
Then the solution is given by 
$$ 1-F(z,t)= [1-F(z,0)]\exp(\int_0^tG(z,s)\,ds).$$
Multiplication with $z^{-1/\theta}$ yields
$$ [1-F(z,t)]z^{1/\theta}= [1-F(z,0)]z^{1/\theta}\exp(\int_0^tG(z,s)\,ds).$$
For a fixed time $s$ the function $G(z,s)$ is monotonically increasing in $z$ and bounded. Hence we can pass to the limit $z\to \infty$ and deduce that the function $F(z,t)$ also has a Pareto tail
for all time $t \in [0,T]$.
\end{proof}

\begin{lem}
Let assumption \ref{a:pareto} be satisfied. Then the growth parameter $\gamma \in \mathbb{R}$ is given by
\begin{align}\label{e:gamma}
\gamma = \theta \int_0^\infty \alpha(\sigma(y))\phi(y) dy.
\end{align}
\end{lem}
\begin{proof}
Recall that $F = \int_0^z f(y,t) dy$ satisfies \re{e:F}. Using that $F(z,t)=\Phi(ze^{-\ga t})$ we deduce that 
\beq
\ga \Phi'(x)x=[1-\Phi(x)]\int_0^x\al(\sg)\phi(y)\,dy.\la{e.Phi_al_x}
\eeq
Dividing this equation by $x^{-\theta}$, passing to the limit $x\to \infty$ and using assumption \ref{a:pareto}, we obtain 
$$-\ga\f 1 \theta k=-k \int_0^\infty \alpha(\sigma(y))\phi(y) dy,$$
and subsequently formula \re{e:gamma}.
\end{proof}

\subsection{Existence}
Next we study the existence of BGP. In the special case of $\al = \alpha_0$ we can show existence of solutions, for
the fully coupled problem we shall provide first results on the existence of $(\gamma, \phi)$ for a given $(v, \sigma)$.

\subsection{Special case: $\al = \alpha_0$}
We reiterate that system \eqref{e:bgp2} decouples in the case of a constant learning function $\alpha = \alpha_0 \in \R^+$. 
Hence we solve equation \re{e.phi2} for $\phi = \phi(z)$ first. Since $\al$ is constant, we have $\sigma=0$ and we can solve the equation \re{e.v2} for $v = v(z)$,
giving us a global solution.

\begin{thm}
Let assumption \ref{a:pareto} be satisfied and $\al = \alpha_0$. Then there exists a unique BGP solution $(\Phi, v, 0)$ and a scaling constant $\gamma$ to \eqref{e:bgp2} given by
\begin{align*}
\gamma =  \alpha_0 \theta \int_{\mathcal{I}} f_0(z) \,dz , \quad\Phi(x) =\f{1}{1+kx^{-1/\theta}}.
\end{align*}
\end{thm}

\begin{proof}
\noindent Note that the initial datum $f_0$ uniquely determines the scaling constant $\gamma$ for the BGP in the case of constant learning function $\alpha  = \alpha_0$. 
Since
$$\int_{\mathcal{I}} f_0(z)\,dz=\int_{\mathcal{I}} f(z) \,dz = \int e^{-\ga t}\phi(ze^{-\ga t})\,dz=\int \phi(y)\,dy,$$
we deduce for $\gamma$ that 
$$\ga = \theta \alpha_0\int \phi(y)\,dy= \theta \alpha_0 \int_{\mathcal{I}} f(z,t) \,dz.$$
For $\al = \alpha_0$ we can write equation \eqref{e.phi2} as 
\beq
-\ga \Phi'(x)x=-\al_0[1-\Phi(x)]\Phi(x).\la{e.Phi}
\eeq
We define  $\tilde \Phi := \f {1}{1 -\Phi}$  i.e.  $ \Phi := 1-\f {1}{ \tilde\Phi}$ and obtain
$$x \ti \Phi'=a(\tilde \Phi -1)$$
with $a  = \frac{\alpha_0}{\gamma}$. This equation has the solution
$$\tilde \Phi = bx^a+1 \qquad \text{and}\qquad\Phi=\f{bx^a}{bx^a+1}.$$
Since we look for solutions which have a Pareto tail we can determine the constant $b$ from the Pareto tail assumption \ref{a:pareto} and obtain
\beq 
\Phi=\f{1}{1+kx^{-1/\theta}}.
\eeq
\noindent In the case $\al = \al_0$, the maximum on the RHS of \eqref{e.v2} is attained at $\sg=0$. 
 In order to solve  equation \re{e.v2} with $\phi$ given and $\sg=0$, we consider $w:=v \phi$ which satisfies
\begin{align}\label{e:w} 
rw+\ga w'x=\phi x+\al_0 \phi \int_x^\infty w\,dy - \al_0 w \int^x_0\phi \,dy.
\end{align}
 We define $W:= \int_x^\infty w\,dy$ and rewrite equation \eqref{e:w} as
 \begin{align*}
 & rw+\ga w'x=\Phi'x+\al_0 \Phi' W - \al_0 \Phi w.
\end{align*}
This gives:
\begin{align*} 
 &(\ga-r)W'-\ga (W'x)'=\Phi'x+\al_0 (\Phi W)'.
 \end{align*} 
 With $H'(x)=x\Phi'(x)$, it follows that
\begin{align}\label{e:W} 
 (\ga-r)W-\ga W'x=H+\al_0 \Phi W+K.
\end{align}
 Equation \eqref{e:W} is a first order ODE with $W(\infty)=0$, which can be solved by considering the solution to the homogeneous equation and then using the variation of constants method.
 Hence we obtain a unique solution $w$, and therefore $v=w/\phi$ .
\end{proof}

\noindent For constant $\alpha$ we are able to prove the following results about the stability of BGPs. 

\begin{thm}
Let assumption \ref{a:pareto} be satisfied, $\al = \alpha_0$ and $f$ denote a solution to the original BMFG problem \eqref{e:bmfg1}. Then the rescaled density $\psi = e^{\gamma t} f(x e^{\gamma t}, t)$ converges pointwise to the BGP
solution  $\phi$, given by \eqref{e:rescalbgp}, as $t\to \infty$. 
\end{thm}
\begin{proof}
The function $\psi(x,t)$ solves 
$$\p_t \psi(x,t)-\ga \psi(x,t) -\ga \p_x \psi(x,t) x =\psi(x,t)\int_0^x\al_0 \psi(x,t)\,dy - \al_0\psi(x,t)\int_x^\infty\psi(x,t)\,dy$$
and therefore its primitive $\Psi(x,t):=\int_0^x\psi(y,t)\,dy$ satisfies
\begin{align}\label{e:Phi}
\p_t \Psi(x,t)-\ga \Psi(x,t) -\ga x \p_x \Psi(x,t)  = -\al_0 [1-\Psi(x,t)]\Psi(x,t).
\end{align}
%Note that we use capitalized letters for the primitives of  $\phi$ and $\vi$ in the following.
We define $U(x^{-1/\theta},t):=x^{1/\theta} \f 1 k\left(\f 1 \Psi -1\right)$  i.e $ \Psi = \f{1}{1+k x^{-1/\theta}U}$ and deduce
\begin{align*}
&\p_t \Psi =-\Psi^2 k x^{-1/\theta} \p_t U,\\
&\p_x \Psi =\Psi^2 k x^{-1/\theta} \f 1 \theta \left( x^{-1/\theta-1}U+x^{-2/\theta-1} \p_y U\right),\\ 
&\Psi (1-\Psi) =\Psi^2 k x^{-1/\theta}  U.
\end{align*}
Hence we can rewrite equation \eqref{e:Phi}  as
$$-\p_t U - \f \ga \theta U-\f \ga \theta y \p_y U= - \al_0 U.$$
Since $\vi$ has a Pareto-tail, $U$ solves the equation
$$\p_t U +\f \ga \theta y \p_y U= 0 \text{ with } U_0(x)=U(x,0)=1.$$
This equation has the solution $U(y,t)=U_0(e^{-\ga t}y)$, hence $U(y,t)$ converges pointwise to 1 for $t\to \infty$.
\end{proof} 

\subsection{Existence for the general model}
We conclude the section by proving existence of a solution $(\gamma, \phi)$, which has a Pareto tail for a given $(\sigma, v)$.
Existence of the fully coupled system is a challenging problem.
\begin{thm}
Let assumption \ref{a:alpha} hold and $\sigma\in C^1([0,\infty))$ denote a given function, which satisfies
\begin{align*}
\sigma(z) = 1 \text{ for } z \in [0, z_0],~\sigma'(z) \leq 0%, ~\lim_{z \rightarrow \infty} \sigma(z) = 0, ~ v'(z) \geq 0 \text{ for all } z \in \mathcal{I}
.
\end{align*}
Then there exists a $\gamma \in \R^+$ and a  solution $\phi \in L^1([0,\infty))$ to equation \eqref{e.phi2}, which has a Pareto tail.
\end{thm}

\begin{proof}
The function $\sigma$ is non-increasing and equal to $1$ in the interval $[0,  x_0]$, hence $\alpha(\sg(z))= \al(1)$ on $[0,  x_0]$. 
We know from the previous subsection that in the case $\alpha = \alpha_0$ there exists a solution $\Phi = \Phi(x)$, $x \in [0,x_0]$ of
the form
\begin{align}\label{e:bgpphi}
\Phi(z)=\f{bx^a}{bx^a+1},
\end{align}
where $a=\al(1)/\ga$ on the interval $[0,  x_0]$. That way we obtain $\Phi(x_0)$.
%%Now we show that there exists a solution on the interval $[x_0, \infty)$, which
%%we glue to \eqref{e:bgpphi} at $x = x_0$. 
%%The primitive $\Phi(x) = \int_0^x\phi(y)dy$ satisfies
%%\beq
%%-\ga \Phi'(x)x= [1-\Phi(x)]\int_0^x\al(\sigma(y)) \Phi'(y)\,dy .
%%\eeq
%%Since $\sigma$ tend to $0$, $\lim_{x \rightarrow \infty} \alpha(\sigma) = 0$ which gives existence of a solution $\Phi$ on the
%%interval $x \in [x_0, \infty)$. 
%At $x_0$ we have 
%$$\Phi_0:=\Phi(x_0)=\f{bx_0^{1/\theta}}{1+bx_0^{1/\theta}}$$

We  rewrite \re{e.Phi_al_x} by integration by parts as 
\beq
\ga \Phi'(x)x=[1-\Phi(x)]\left[\Phi(x)\al(\sg(x))-\int_0^x\Phi(y)\f{d}{dy}\al(\sg(y))\,dy\right]\la{e.Phi_ode}
\eeq
and can obtain a solution on $[x_0, \infty)$  starting from $\Phi(x_0)$  using Picard-Lindel\"of.  Next we check that the Pareto-tail condition is satisfied.
To do so we change variables to $\ti x :=x^{-1/\theta}$ and define $\Phi(x)=:1-K(\ti x)\ti x$. The existence of the limit 
$$\lim_{x \to \infty} \f{1-\Phi(x)}{x^{-1/\theta}} $$ is equivalent to the  existence of $\lim_{\ti x \to 0} K(\ti x).$
In this new variable equation \re{e.Phi_al_x} reads as 
 $$K'\ti x=-K \int_0^{\ti x}  \al(\tilde \sg)(K\xi)' \,d\xi,$$
where $\ti \sg$ is $\sg$ transformed in the new variable $\ti x$. Since $(K\ti x)'$ is non-negative we can estimate the integral on the right hand side  
\begin{align*}
&K'\ti x\ge-K \int_0^{\ti x} \bar \al (K\xi)' \,d\xi=-K \bar \al  K\ti x.
\end{align*}
Then it follows that 
$$
-\f{K'}{K^2} \le \bar \al\quad \text{ and therefore } \quad K(0) \le \f{K(\ti x_0)}{1- \bar \al \ti x_0 K(\ti x_0)}.$$ Since $\Phi(x)\to 1$ for $x\to \infty$, we have $\ti x_0 K(\ti x_0)\to0$ for $\ti x_0 \to 0$.
This means that the limit 
$ \lim_{x \to \infty} \f{1-\Phi(x)}{x^{-1/\theta}} $ exists.\\
The proof follows by iterating between $\Phi$ and $\gamma$. For $\theta <1$, the derivative $\Phi'$ is bounded. So we obtain a fixed point $(\ga,\phi)$ satisfying equations  
\eqref{e.phi2} and \re{e:gamma}.
\end{proof}

\section{Numerical simulations}\label{s:numerics}
\noindent In this section we present an iterative scheme to solve \eqref{e:bmfg1} numerically. First we illustrate
the formation of blow-up solutions for the Boltzmann type model \eqref{e:boltz}. Then we compare the numerical solution of our scheme with the
BGP calculated using the code provided by Lucas and Moll. For additional information on the construction of solutions and the numerical solver for the BGP we refer
to \cite{LM2013}. Finally we study the stability of BGP solutions with respect to initial perturbations.

\subsection{Numerical scheme}
We consider a bounded domain $\mathcal{I} = [0, \bar{z}]$, where $\bar{z}$ denotes to the maximum knowledge level.
The spatial discretization corresponds to $N$ logarithmically spaced intervals. The temporal domain is split into equidistant time steps of size $\Delta t$.
 Let $z_i$ denote the $i-th$ grid point and $t_k = k \Delta t$ the k-th time step. 
We set the initial agent distribution to $f(x,0) = f_0(x)$ to 
\begin{align*}
  f_0(x) = \frac{k}{\theta} x^{-\frac{\theta+1}{\theta}}e^{-kx^{\frac{1}{\theta}}}.
\end{align*}
Note that in this case the cumulative distribution function of $F_0(z) = \int_0^z f_0(y) dy$ has a Pareto tail, i.e. assumption \ref{a:pareto} is satisfied.
Let $V(z,T) = 0$ be the terminal condition of the optimal strategy in the iterative solver. Then the numerical results of the full model are based on the following iterative procedure:
\begin{enumerate}
\item Calculate the solution $f = f(z,t)$ of the Boltzmann type equation \eqref{e:boltz} by updating the solution via
an explicit in time discretization:
\begin{align*}
  f(z_i, t_{k+1}) = f(z_i, t_k) + \Delta t &\left[-\alpha(s(z_i,t_k)) f(z_i, t_k) \int_{z_i}^{\bar{z}} f(y,t_k) dy + \right.\\
&\left. f(z_i,t_k) \int_0^{z_i} \alpha(s(y,t_k)) f(y,t_k) dy \right],
\end{align*}
where the integrals on the right hand side are evaluated using the trapezoidal rule.
\item Solve the Hamilton Jacobi Bellman equation \eqref{e:bellmann}, by first calculating the maximum
via the optimality conditions. Then determine the new optimal velocity using the updated density 
distribution $f = f(z,t)$  and the function $s$.\\
In particular  
\begin{align*}
 \max_{s(z) \in \mathcal{S}}\left((1-s) z + \alpha(s) \int_{z}^{\bar{z}} (V(y,t)-V(z,t)) f(y,t) dy\right)
\end{align*}
is given by
\begin{enumerate}
\item $z=0$: Then $s=1$ since $\alpha = \alpha(s)$ is a strictly monotone function.
\item $z=\bar{z}$: Then the integral is equal to zero, hence $s=0$ gives the maximum.
\item $z \in (0,\bar{z})$: Then $s = \min\left((\frac{z}{\alpha_0 nB})^{\frac{1}{n-1}},1\right)$, hence we cut-off $s$ if it lies outside the
interval $[0,1]$.
\end{enumerate}

\item Go to (1) until convergence.
\end{enumerate}

\subsection{Numerical simulations of the Boltzmann type equation for a given $\alpha$}
In our first example we study mass accumulation in the case of compactly supported initial data $f(z,0)$, cf. Lemma \ref{l:deltadirac}. 
We assume that the maximum knowledge level is denoted by $\bar{z} = 1$,
and choose an initial agent distribution of
\begin{align*}
f(z,0) = 
\begin{cases}
2 &\text{ for all } z \leq 0.5 \\
0 &\text{ otherwise}.
\end{cases}
\end{align*}
We set $\alpha = 1-z$, i.e. individuals with the lowest knowledge devote all their time to
'interactions', those with the maximum knowledge do not interact at all.
Figure \ref{f:delta} illustrates the formation of a Dirac at $z=0.5$ for symmetric and non-symmetric meetings. Note that the lines correspond to $f$ at different discrete time steps. Furthermore we observe 
that the formation of the Dirac happens much faster in case of symmetric meetings, i.e. $\beta=1$.
\begin{figure}\label{f:delta}
\begin{center}
\subfigure[$\beta=0$]{\includegraphics[width=0.45 \textwidth]{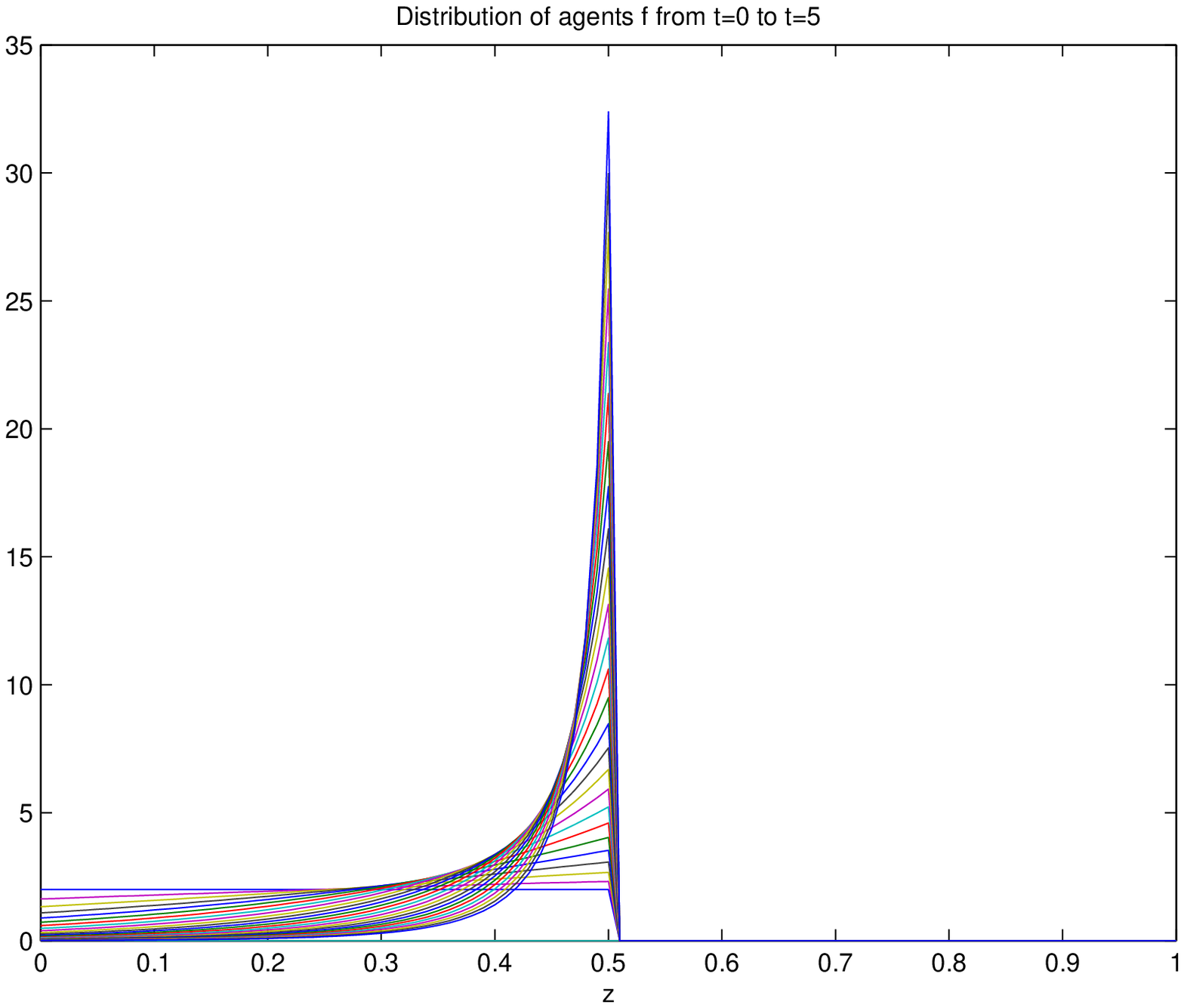}}\hspace*{0.5cm}
\subfigure[$\beta=1$]{\includegraphics[width=0.45 \textwidth]{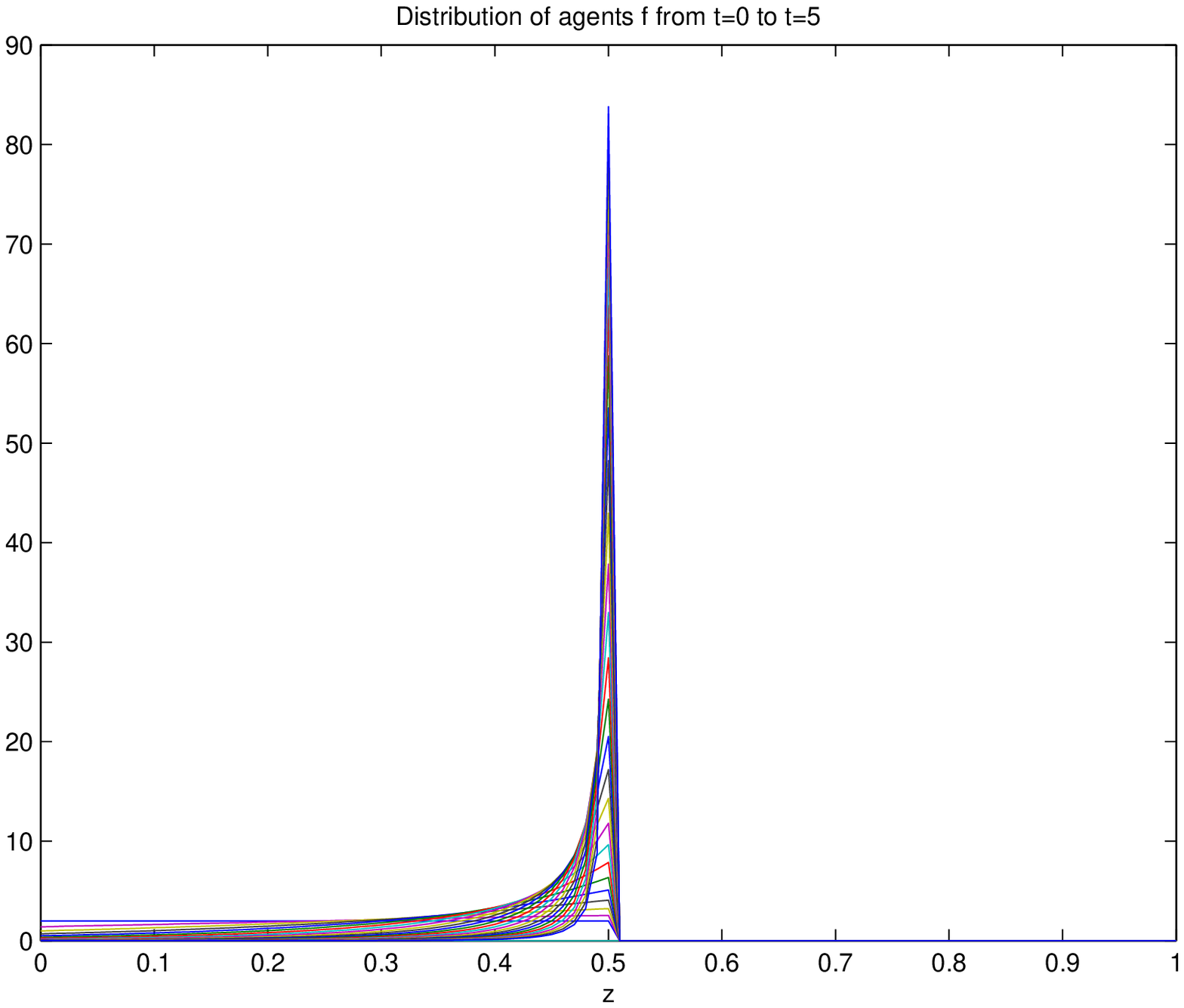}}
\caption{Distribution of agents $f = f(z,t)$ at different times starting at $t=0$ up to $t=5$.}\label{f:delta}

\end{center}
\end{figure}

\subsection{Numerical simulations of the full Boltzmann mean field game model}
Next we compare our results with the numerical simulations of the balanced growth path solutions by Lucas and Moll. 
We choose the same simulation parameters in \eqref{e:alpha} and \eqref{e:paretoF}:
\begin{align*}
\alpha_0 = 0.0849, ~ n = 0.3, ~\theta = 0.5, k = 0.05 \text{ and } r = 0.06.
\end{align*}
To compare the results of the two numerical solvers we fix a sufficiently large final time $T$, i.e. $ t \in [0,200]$ and rescale the computational domain $\mathcal{I}$ 
using the parameter $\gamma$ determined from the BGP simulations of Lucas and Moll. The simulation results for
$k = 400$ time steps and $N = 1001$ discretization points are depicted in Figure \ref{f:trans_vs_bgp}.
\begin{figure}
\begin{center}
\subfigure[Transient vs. BGP]{\includegraphics[width=0.45\textwidth]{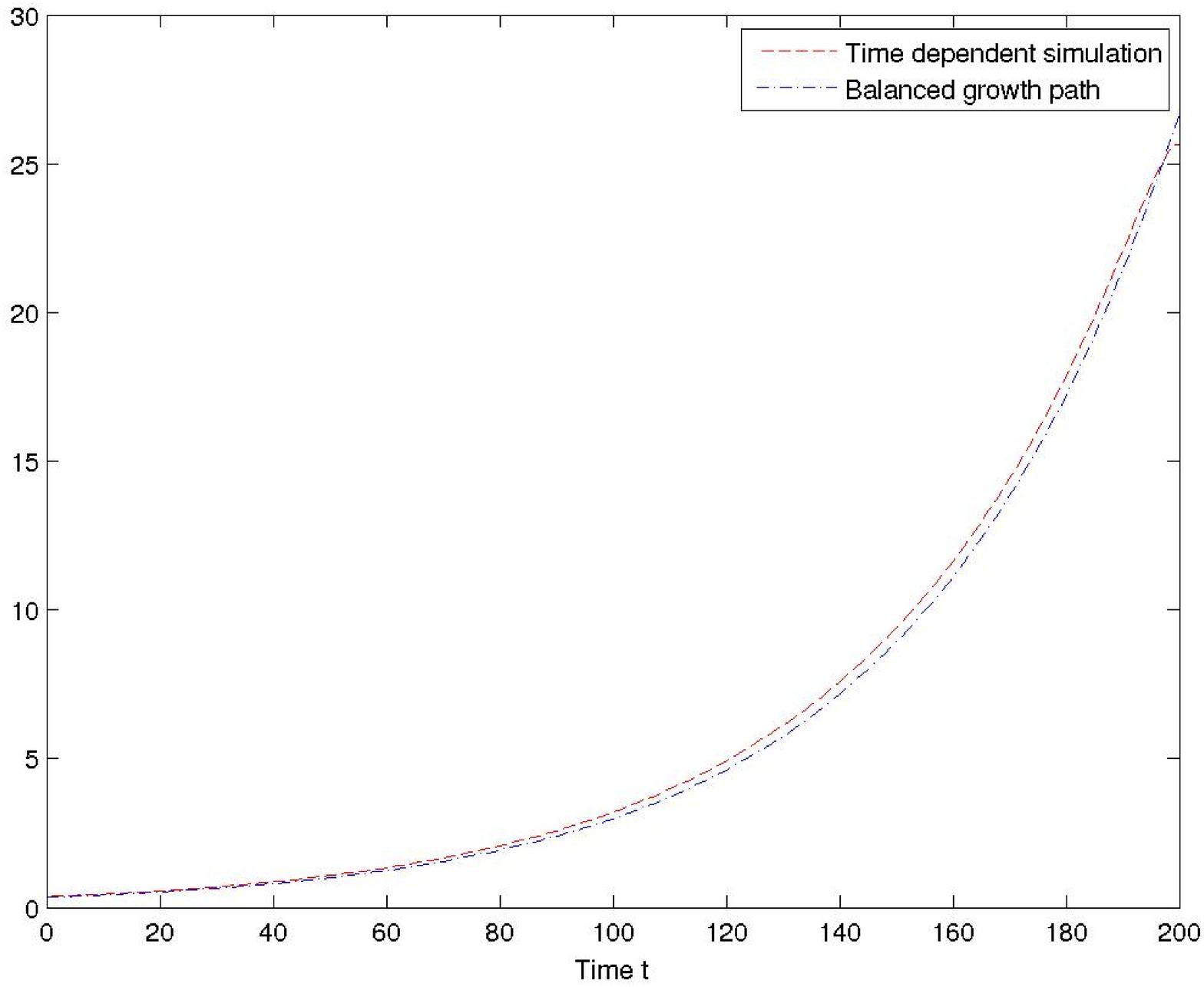}\label{f:trans_vs_bgp}}
\subfigure[Linear growth]{\includegraphics[width=0.45\textwidth]{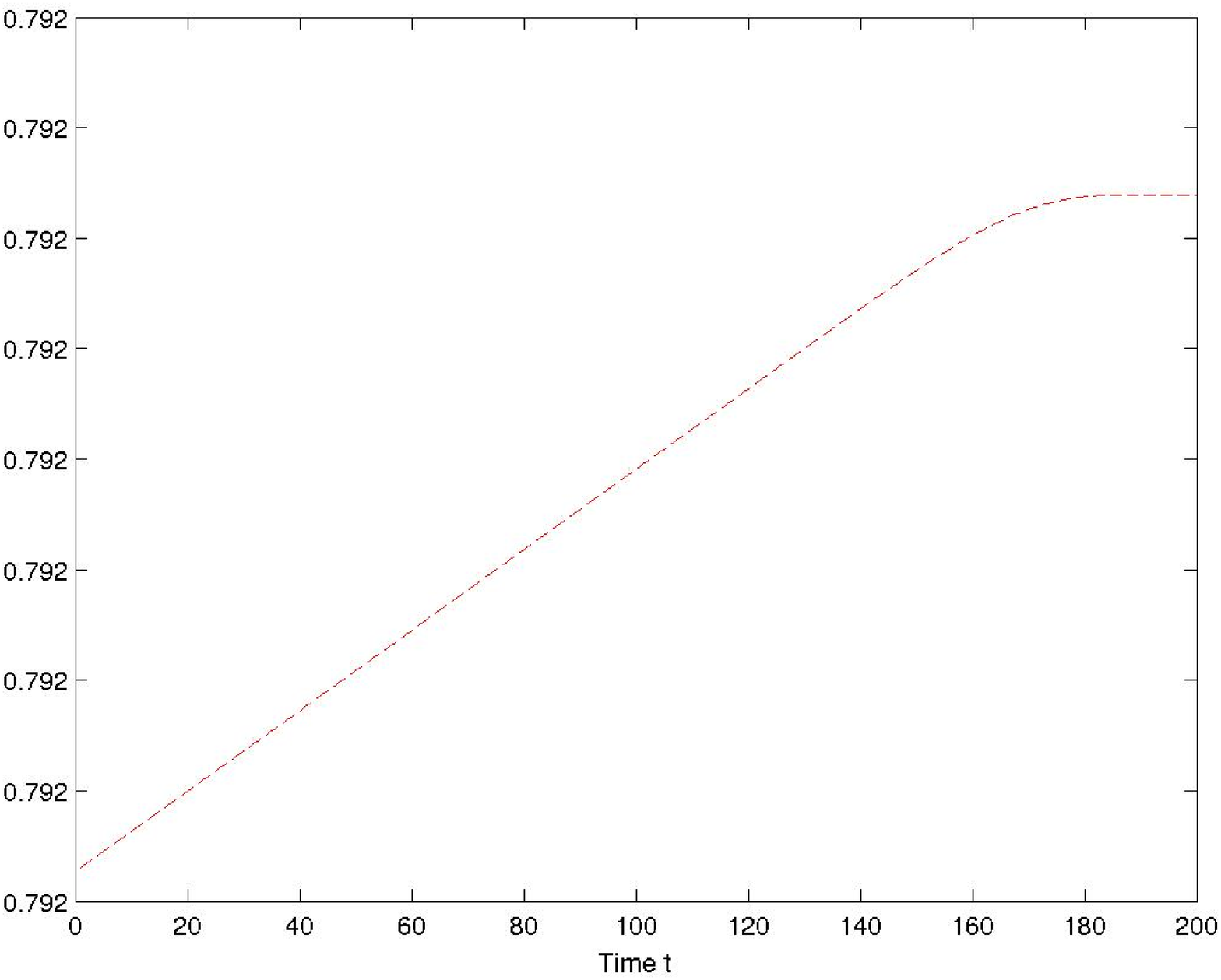}\label{f:theta_small}}
\caption{Evolution of the production function $Y = Y(t)$ in time for different choices of $n$ and $\theta$ }
\end{center}
\end{figure}
Note that we can not expect the existence of BGP solutions in general. If we choose for example
\begin{align*}
\alpha_0 = 0.0849, ~ n = 0.9, ~\theta = 0.1, k = 0.05 \text{ and } r = 0.06,
\end{align*}
i.e. a larger value of $n$ and a smaller value of $\theta$ then the BGP solver of Lucas and Moll is not converging.
The simulation results in Figure \ref{f:trans_vs_bgp} reveal the reason why this is the case. Here the production
function $Y = Y(t)$ is growing linear in time, hence the ansatz proposed by Lucas and Moll is not satisfied. 

\subsection{Stability of balanced growth path}
In our final example we illustrate the stability of balanced growth path with respect to small perturbations of the
initial data. We consider a perturbation of the initial distribution of the form
\begin{align*}
  f_0^p(z) = f_0(z) +  0.1 (1-z) \sin(25 \pi z) \chi_{[0.1,1]},
\end{align*}
which corresponds to a perturbation on the interval $[0.1,1]$ that does not change the overall mass, see Figure \ref{f:per_f0}. Note that the 
initial datum still has a Pareto tail. We set $\alpha_0$, $n$, $k$, $\theta$ and $r$ to the same values as in the previous
example and solve the system on the time interval $t \in [0,250]$ using $500$ equidistant time steps.
Figure \ref{f:bgp} compares the evolution of balanced growth path for the
corresponding unperturbed initial datum, with the transient simulation. We observe that this perturbation does not
change the long time behavior of the production rate. Note that the difference of the two solutions at $t=250$ can be
explained by the terminal condition $V(z,t=250) = 0$ for the transient simulation.

\begin{figure}
\begin{center}
\subfigure[Perturbed initial data]{\includegraphics[width=0.45\textwidth]{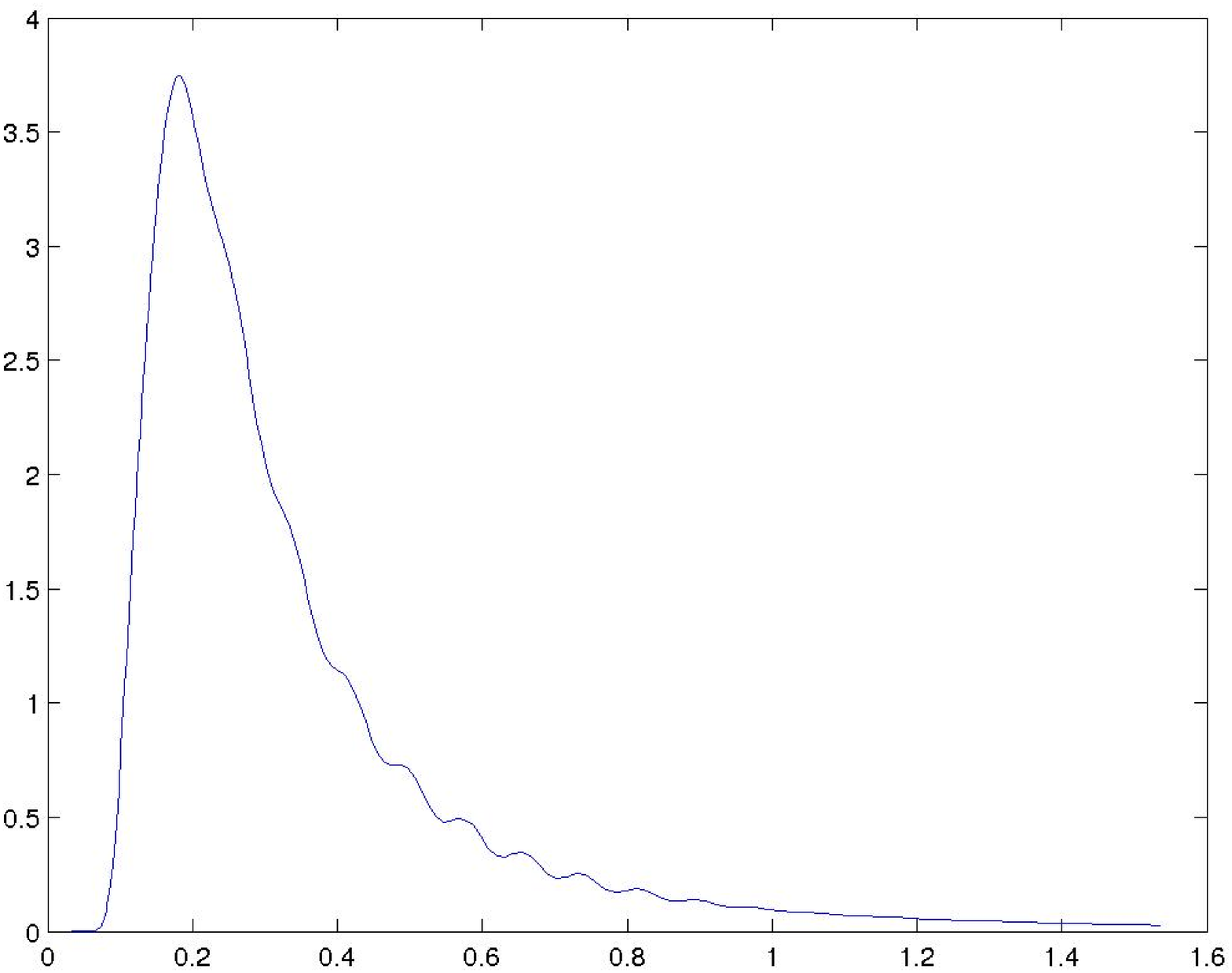}\label{f:per_f0}}
\subfigure[Evolution of the production rate $Y$]{\includegraphics[width=0.45\textwidth]{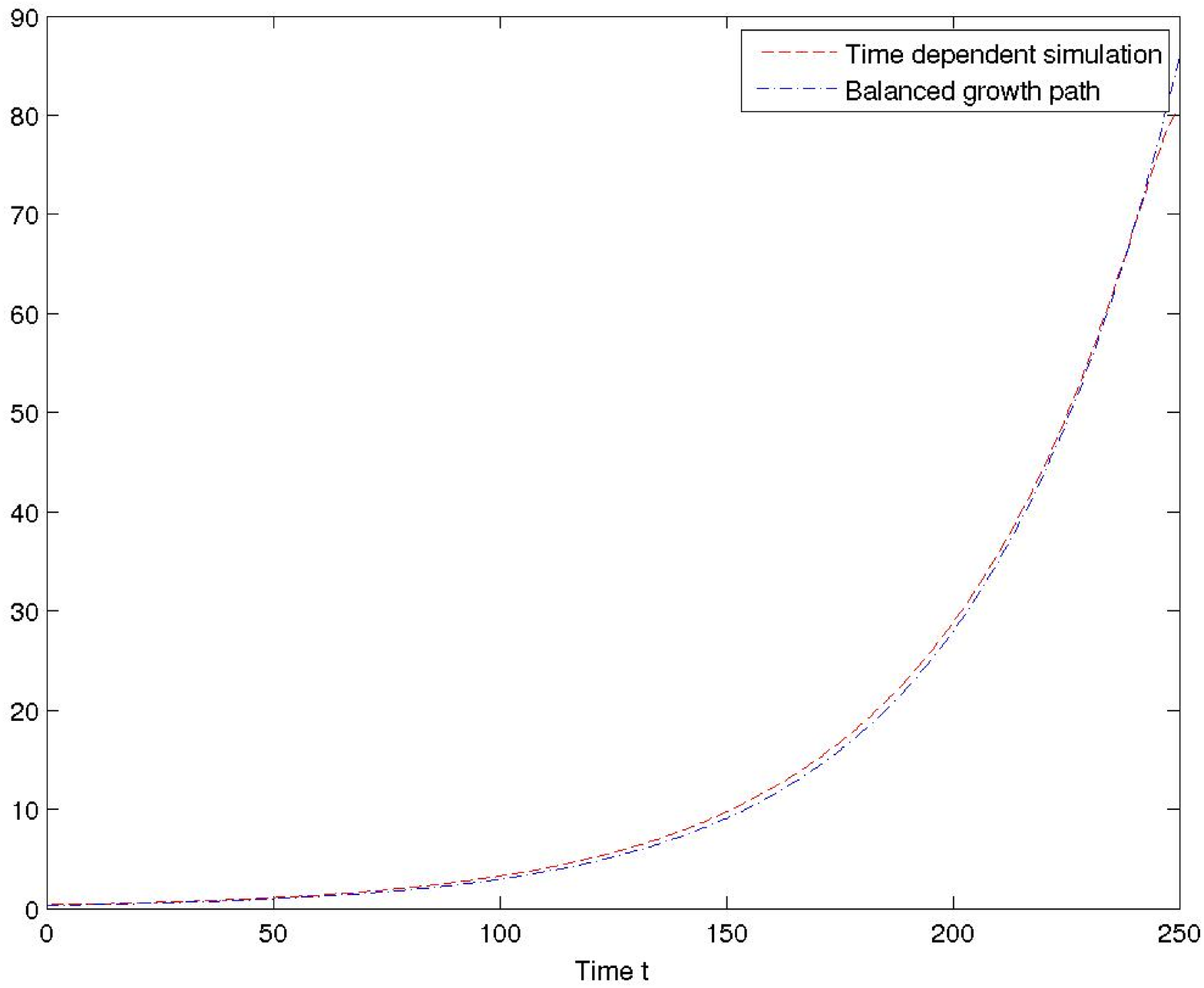}}
\caption{Stability of BGP solutions}\label{f:bgp}
\end{center}
\end{figure}

\section{Conclusion}
\noindent In this paper we present first analytic results as well as numerical simulations for a novel Boltzmann mean field game model for 
knowledge growth proposed by Lucas and Moll \cite{LM2013}. In this model the distribution of individuals with respect to the knowledge
level evolves according to a Boltzmann equation. Collisions correspond to knowledge exchange, the individual interaction rate is determined
by maximizing the individual productivity (given the common knowledge of the distribution of all other agents). This gives to a coupled
system of a Boltzmann equation and a Hamilton-Jacobi-Bellman equation. Knowledge growth is an inherent nature of the model, which is also
reflected in the analytic results. The value function of the HJB equation is growing linearly in $z$, hence we can only provide
local in time existence. Balanced growth path solutions, which also correspond to exponential growth of the production function in time, illustrate
this nature as well - although we are not able to provide existence results of these special solutions in a general situation, we provide
first results on their existence and stability in the case of a special interaction function. \\

\noindent This summary gives indications about several open analytic problems which shall be addressed in the near future, e.g. existence and stability of BGP solutions for the fully coupled system.
Another point of interest corresponds to the generalization of the model. For example by considering a common
noise via an additional diffusion term, more general interaction laws for knowledge growth or knowledge decay caused by forgetting information.

\section*{Acknowledgments}
\noindent MTW acknowledges financial support from the Austrian Academy of Sciences \"OAW via the New Frontiers Group NST-001. 
This research was funded in part by the French ANR blanche project Kibord: ANR-13-BS01-0004.\\
The authors want to thank Benjamin Moll for his support and valuable comments while preparing this manuscript.

\bibliographystyle{abbrv}
\bibliography{mfg_boltzmann}

\end{document}